\documentclass{article}%

\usepackage{graphicx}
\usepackage{amsmath,amssymb,amsfonts}%
\usepackage{theorem}%
\usepackage{color}%
\usepackage{hyperref}
\usepackage[font={small, it}]{caption}
\usepackage{float}

\theoremstyle{change}%
\newtheorem{definition}{Definition:}[section]%
\newtheorem{proposition}[definition]{Proposition:}%
\newtheorem{theorem}[definition]{Theorem:}%
\newtheorem{lemma}[definition]{Lemma:}%
{\theorembodyfont{\rmfamily}\newtheorem{remark}[definition]{Remark:}}%
{\theorembodyfont{\rmfamily}}%

\newenvironment{proof}
{{\bf Proof:}}
{\qquad \hspace*{\fill} $\Box$}%

\newcommand{\fh}{\mathfrak{h}}%
\newcommand{\fl}{\mathfrak{l}}

\newcommand{\tr}{\operatorname{tr}}%

\newcommand{\inner}{\operatorname{int}}%
\newcommand{\rme}{\mathrm{e}}%

\newcommand{\CC}{\mathcal{C}}%
\newcommand{\OC}{\mathcal{O}}%
\newcommand{\XC}{\mathcal{X}}%
\newcommand{\DC}{\mathcal{D}}%

\newcommand{\HB}{\mathbb{H}}%
\newcommand{\T}{\mathbb{T}}%
\newcommand{\N}{\mathbb{N}}%
\newcommand{\R}{\mathbb{R}}%
\newcommand{\Z}{\mathbb{Z}}%

\begin{document}
	
	\title{Linear control systems on the homogeneous spaces of the Heisenberg group}

\author{Adriano Da Silva \\
		Departamento de Matem\'atica,\\Universidad de Tarapac\'a - Iquique, Chile.
		\and
		Okan Duman\; and\; Ey\"up Kizil \\
		Department of Mathematics \\
		Yildiz Technical University - Istanbul, Turkey.\\
}
\date{\today }
\maketitle

\begin{abstract}
 Let $\HB$ denote the 3-dimensional Heisenberg Lie group. The present paper classify all possible linear control systems on the homogeneous spaces of $\HB$ through its closed subgroups and expose a detailed study on the control behavior (controllability property and control sets) of a particular dynamics evolving on a non simply connected homogeneous (state) space of dimension two.
 
\end{abstract}

	

\section{Introduction}
Linear control systems on Lie groups appear as a natural generalization to connected Lie groups of the well-known class of linear systems on Euclidean spaces (See Ayala-Tirao paper in \cite{Ayala} and also the paper \cite{Markus} by L. Marcus on classical matrix Lie groups). Here, we consider the 3-dimensional Heisenberg Lie group $\HB$ together with its closed subgroups (discrete subgroups included and normal subgroups excluded \footnote{We exclude normal subgroups of $\HB$ since otherwise the corresponding homogeneous spaces receive a Lie group structre and linear control systems on such state
spaces has been already studied in a series of papers. See, \cite{Ayala}, 
\cite{VA1},\cite{VA2},\cite{VA3},\cite{VAP}}) to form its homogeneous spaces and classify on such state spaces all possible linear control systems, which is not a trivial task. The motivation comes from a recent result by P. Jouan in \cite{Jouan} that emphasizes a quite interesting connection between a control affine system on a manifold and a linear system either on a Lie group or a homogeneous space. More precisely, a control affine system on a manifold is equivalent by mean of a diffeomorphism to a linear control system on a Lie group or a homogeneous space if and only if the vector fields that describe the system are complete and generate a finite dimensional Lie algebra. It follows that one might find in some suitable context a control system on a manifold that is equivalent to a linear control system on a homogeneous space of $\HB$. Hence, we find it convenient to give in this article complete characterization of all possible linear systems on homogeneous spaces of $\HB$ and deal with dynamical properties of such systems as a concrete case. To have such linear systems on homogeneous spaces of the Heisenberg group we have to determine explicitly the conditions that guarantee well-defined induced dynamics on various quotient spaces. Since this requires a certain invariance criteria of subgroups of $\HB$ under the flow of the drift (i.e., a linear vector field) of the original dynamics (See Proposition 4.1., \cite{Jouan}) we start with listing these conditions first to obtain the induced or projectable drift and control vectors (i.e., left-invariant vector fields) on the corresponding homogeneous spaces. 

It should be noted that determining controllability property, characterizing eventual topological properties of control sets (i.e.
 regions of approximate controllability) of all of these systems becomes highly non-trivial job. For example, even on low dimensional groups the properties of control sets for such dynamics on Lie groups and homogeneous spaces might be quite different (See \cite{VAAA} and \cite{VA1} ). Hence, we select among others only a certain 1-dimensional subgroup of $\HB$ to form a 2-dimensional non-simply connected homogeneous space as the state space and consider a particular linear system on it, for which we are able to fully characterize the control sets. A much more detailed and complex work will be left to a future work.
 
 The paper is divided into 6 sections. In Section 2, we mention some generalities in control setting to facilitate a better understanding of the rest of the manuscript. In section 3, we fix the format on which the whole exposition is based on. More precisely, rather than the group of upper triangular matrices with only 1s in the main diagonal we prefer to interpret the Heisenberg group $\HB$ as the cartesian product $\mathbb{R}^{2}\times \mathbb{R}$ and express all the necessary arguments such as the group multiplication, invariant and linear vector fields and their Lie brackets, etc to be in accordance with this format. With that, we are able to obtain, up to isomorphism, any closed subgroup of $\HB$ with dimension 0,1 and 2. We conclude the section with a brief resume of linear control systems on Lie groups which will be used through the subsequent sections. Section 4 focuses on a  certain invariance criteria of subgroups (that is, discrete and non normal subgroups) of $\HB$ under the flow of a linear vector field. By using the classification of closed subgroups in Section 3, we are able to obtain in Proposition 4.4 the conditions a linear vector field should satisfy in order to achieve the desired invariance condition of the subgroups under consideration. In Section 5, we define what we mean by a linear control system on homogeneous spaces of $\HB$ and list, up to equivalence, all possible such systems. In the last section 6, we constrain our attention to a particular linear system on a certain homogeneous space of $\HB$ to characterize topologically its control sets and also controllability. See Lemma 6.1 and Theorem 6.2

\section{Preliminaries}
Let $M$ be a finite dimensional smooth manifold and let $\mathbb{R}^{m}$
denote the $m$-dimensional Euclidean space. Given a compact convex subset $%
\Omega \subset \mathbb{R}^{m}$ satisfying $0\in $ int $\Omega $, we mean by
a control-affine system evolving on $M$ the following (parametrized) family
of ordinary differential equations 
\begin{equation*}
\Sigma _{M}:\quad \dot{x}(t)=f_{0}(x(t))+\sum_{j=1}^{m}\omega_{j}(t)f_{j}(x(t)),%
\quad \omega\in \mathcal{U}
\end{equation*}%
where $f_{0},f_{1},\ldots ,f_{m}$ are smooth vector fields defined on $M$
and the control parameter $\omega=\left(\omega_{1},\ldots, \omega_{m}\right) $ belongs to
the set $\mathcal{U}$ of the piecewise constant functions such that $\omega(t)\in
\Omega$. We also assume w.l.o.g. $m\leq \dim M$ and that the set $\{f_1, \ldots, f_m\}$ is linearly independent in the set of the smooth vector fields on $M$.

For an initial state $x\in M$ and $\omega\in \mathcal{U}$, the solution of $%
\Sigma _{M}$ is the unique absolutely continuous curve $t\mapsto \phi (t,x,\omega)
$ on $M$ satisfying $\phi (0,x, \omega)=x$. Associated to $\Sigma _{M}$ we have
for a given $x\in M$ the positive orbit at $x$ as follows:%
\begin{eqnarray*}
\mathcal{O}^{+}(x) &=&\{\phi (t,x,u):t\geq 0,\omega\in \mathcal{U}\}
\end{eqnarray*}%
We say that $\Sigma _{M}$ satisfies the \emph{Lie algebra rank condition}
(abrev. LARC) if $\mathcal{L}(x)=T_{x}M$ for all $x\in M,$ where $\mathcal{L}
$ denotes the smallest Lie algebra of vector fields containing $\Sigma _{M}$%
. The system $\Sigma _{M}$ is said to \emph{controllable} if $M=\mathcal{O}%
^{+}(x)$ for all $x\in M$.

Next we introduce the concept of control sets encountered \cite{FW}. 

\begin{definition}
A set $\mathcal{C}\subset M$ is a control set of $\Sigma _{M}$ if it is
maximal, w.r.t. set inclusion, with the following properties:
\begin{enumerate}
\item $\forall x\in \mathcal{C}$, there exists a control $u\in \mathcal{U}$
such that $\phi \left( \mathbb{R}^{+},x,u\right) \subset \mathcal{C}$;
\item It holds that $\mathcal{C}\subset \mathrm{cl}~\mathcal{O}^{+}(x)$ for
all $x\in \mathcal{C}$.
\end{enumerate}
\end{definition}

Let $N$ be another smooth manifold and
\begin{equation*}
\Sigma _{N}:\quad \dot{y}(t)=g_{0}(y(t))+\sum_{j=1}^{m}\omega_{j}(t)g_{j}(y(t)),%
\quad \omega\in \mathcal{U}
\end{equation*}%
a control-affine system on $N$.

\begin{definition}
\label{defi:conjugado}If $\varphi :M\rightarrow N$ is a smooth map, we say
that $\Sigma _{M}$ and $\Sigma _{N}$ are $\varphi $-conjugated if their
respective vector fields are $\varphi $-conjugated, that is, 
\begin{equation*}
\varphi _{\ast }\circ f_{j}=g_{j}\circ \varphi
\end{equation*}%
for each $j\in \{0,1,\ldots ,m\}$. If such a $\varphi $ exists, we say that $%
\Sigma _{M}$ and $\Sigma _{N}$ are conjugated. In particular, if $\varphi $
is a diffeomorphism, then $\Sigma _{M}$ and $\Sigma _{N}$ are called
equivalent systems.
\end{definition}

It follows that equivalent systems preserve controllability, topological
properties of control sets and positive (or negative) orbits.

Since in the sequel we also consider control-affine systems on a connected
Lie group (and hence its corresponding homogeneous space) we find it
convenient to provide some basic definitions and facts involving Lie groups
and their Lie algebras.

\begin{definition}
A vector field $\mathcal{X}$ on a connected Lie group $G$ is linear if its
flow $\{\varphi _{t}\}_{t\in \mathbb{R}}$ is a 1-parameter subgroup of $%
\mathrm{Aut}(G)$, the group of all automorphisms of $G$.
\end{definition}

It is well known that a linear vector field on a connected Lie group $G$ is
complete and one can always associate to such a vector field a derivation $%
\DC=-ad(\mathcal{X})$ of the corresponding Lie algebra $\mathfrak{g}$ of $G$.
Recall that a Lie algebra derivation $D$ of $\mathfrak{g}$ is a linear map
on $\mathfrak{g}$ satisfying the Leibnitz rule, that is, $%
\DC[X,Y]=[\DC X,Y]+[X,\DC Y]$ for every $X,Y\in \mathfrak{g}$. Although the converse
does not occur in general we have for a connected and simply connected Lie
group $G$ that given a derivation $\DC$ of the Lie algebra $\mathfrak{g}$ of $G
$, there exists a linear vector field associated to $\DC$ through the formula $%
(d\varphi _{t})_{e}=e^{t\DC},$ $\forall t\in \mathbb{R}$, where $\varphi_{t}$
stands for the flow. In particular, $\varphi _{t}(\exp Y)=\exp (e^{t\DC}Y)$
for every $Y\in \mathfrak{g}$ and $t\in \mathbb{R}$.
		
		\section{The Heisenberg group and its homogeneous spaces}
		
		Throughout the exposition, we let $\HB$ denote the 3D Heisenberg (Lie) group and $\fh$ its Lie algebra.



For simplicity, we will consider the Heisenberg group as $\HB=\mathbb{R}^{2}\times \mathbb{R}$, with product given by 
\begin{equation*}
(\mathbf{v}_{1},z_{1})\ast (\mathbf{v}_{2},z_{2})=\left(\mathbf{v}_{1}+\mathbf{v}%
_{2},z_{1}+z_{2}+\frac{1}{2}\left\langle \mathbf{v}_{1}, \theta\mathbf{v}_{2}\right\rangle\right), \hspace{.5cm}\mathbf{v}_{i}\in\R^2, z_i\in\R, i=1, 2,
\end{equation*}%
where $\left\langle \cdot ,\cdot \right\rangle $ stands for the standard
inner product in $\mathbb{R}^{2}$ and $\theta$ stands for the 
counter-clockwise rotation of $\frac{\pi}{2}$-degrees.



The Lie algebra $\mathfrak{h}$ $=\mathbb{R}^{2}\times \mathbb{R}$ of $\HB$ is equipped with the Lie bracket%
\begin{equation*}
\lbrack (\mathbf{\zeta }_{1},\alpha _{1}),(\mathbf{\zeta }_{2},\alpha
_{2})]=(\mathbf{0},\left\langle \mathbf{\zeta }_{1}, \theta\mathbf{\zeta }%
_{2}\right\rangle ), \hspace{.5cm}\zeta_{i}\in\R^2, \alpha_i\in\R, i=1, 2,\text{.}
\end{equation*}

One of the usefulness of defining the Heisenberg group and its associated algebra as previously, instead of the usual matrix version, is that for this setup the exponential map $\exp :\mathfrak{h}\rightarrow \HB$ is reduced to the identity map on $\HB$. In particular, every connected subgroup $L\subset \HB$ is identified with its Lie subalgebra.

Given a subgroup $L\subset \HB$, we denote by $L_{0}$ the connected component
containing the identity element of $\HB$ and simply call it \emph{the identity
component}, as usual. Note that the identity component $L_{0}$ is a closed
normal subgroup of $L$ and has the same Lie algebra as $L$.\ The other
components are given by the cosets $g\ast L_{0}=L_{0}\ast g$ of $L$.

By our previous setup, it is not hard to see that a typical derivation $\DC$ of the Lie algebra $\mathfrak{%
h}$ of $\HB$ in its matrix form (w.r.t. the standard basis) is given by 
\begin{equation*}
\DC=\left(\begin{array}{cc}
		A & 0\\ \eta^{\top} & \tr A
	\end{array}\right),\hspace{.5cm}\mbox{ where } A\in \mathfrak{gl}(2)\mbox{ and }\eta \in \mathbb{R}^{2}.
\end{equation*}%
Since Lie
algebra derivations are closely connected with Lie algebra automorphisms we
also find it useful to give explicit face of an automorphism by the
following matrix%
\begin{equation*}
\left( 
\begin{array}{cc}
P & 0 \\ 
\eta^{\top}  & \det P%
\end{array}\right) , \hspace{.5cm}\mbox{ where } P\in \mathrm{Gl}(2)\mbox{ and }\eta \in \mathbb{R}^{2}.
\end{equation*}%

\begin{remark}
    In the previous we are assuming that $\R^n=M_{n\times 1}(\R)$. Such identification will be very useful ahead.
\end{remark}

Since the Lie algebra $\mathfrak{h}$ of $\HB$ can be seen as the set of left-invariant
vector fields on $\HB$,  we give below a usual expression of such a vector field
which is notationally appropriate in the present context. Hence, if we pick
a point $g=(\mathbf{v},z)\in \HB$, and an element $B=(\zeta ,\alpha)\in\fh$, the left-invariant vector field on $\HB$ is defined
via the vector space structure by 
\begin{equation*}
B(g)=\left(\zeta ,\alpha +\frac{1}{2}\langle \mathbf{v}, \theta\zeta \rangle\right)\text{.}
\end{equation*}%
It then follows that given a derivation $\DC$ of $\mathfrak{h}$ as above, one
might immediately associate to $\DC$ the linear vector field $\mathcal{X}$ on $\HB$
by $[B,\mathcal{X}]=\DC B$ for every $B\in \mathfrak{h}$. Hence we might write $%
\mathcal{X}$ through the matrix multiplication as follows:%
\begin{equation*}
\mathcal{X}(g)=\left( 
\begin{array}{cc}
A & 0 \\ 
\eta^{\top}  & \mathrm{tr}A%
\end{array}%
\right) \left( 
\begin{array}{c}
\mathbf{v} \\ 
z%
\end{array}%
\right) =\bigg(A\mathbf{v}~,~\langle \eta^{\top} ,\mathbf{v}\rangle +z~\mathrm{tr}A%
\bigg), \hspace{.5cm}g=(\mathbf{v},z)\in \HB.
\end{equation*}%

Let $B\in \mathfrak{gl}(2)$ and let us define the following operator 
\begin{equation*}
\mathbf{\Lambda }^{B}:\mathbb{R}\times \mathbb{R}^{2}\longrightarrow \mathbb{%
R}^{2},\quad \mathbf{\Lambda }_{t}^{B}(\eta )=\int_{0}^{t}e^{s B^{\top}}\eta ds\text{%
.}
\end{equation*}%
It then follows at once that using such an operator we get for $\DC$ that
\begin{equation*}
e^{t \DC}=\left( 
\begin{array}{cc}
e^{tA} & 0 \\ 
\left( e^{t\cdot\mathrm{tr}A}\mathbf{\Lambda }_{t}^{(A-\mathrm{tr}A\cdot I_2)}\eta
\right) ^{\top } & e^{t\cdot \mathrm{tr}A}%
\end{array}%
\right),
\end{equation*}
where in the previous $I_2$ stands for the $2\times 2$ identity matrix. As a consequence, the flow $\varphi _{t}$ induced by $\mathcal{X}$ is given by 
\begin{equation*}
\varphi _{t}(\mathbf{v},z)=\left( e^{tA}\mathbf{v},\left\langle \rme^{t\cdot\mathrm{tr}A}\mathbf{\Lambda }_{t}^{(A-\mathrm{tr}A\cdot I_2)}\eta ,\mathbf{v}\right\rangle
+z\rme^{t~\mathrm{tr}A}\right) \text{.}
\end{equation*}
			
\subsection{Linear control systems on $\HB$}

Before we mention linear control systems on homogeneous spaces of $\HB$ we give first a brief description of a linear control system on $\HB$ since this
is intimately related with that on the corresponding coset spaces of $\HB$. Hence, let $\Omega$ be a compact subset of $\mathbb{R}^3$. By a linear control system (abrev. LCS) on $\HB$ we understand a system of the form 
\begin{equation*}
\Sigma _{\HB}:\quad \dot{(\mathbf{v},z)}=\mathcal{X}(\mathbf{v},z)+\omega_1 B_1(\mathbf{v%
},z)+\omega_2B_2(\mathbf{v%
},z)+\omega_3B_3(\mathbf{v%
},z)
\end{equation*}%
where $\omega=(\omega_1, \omega_2, \omega_3)\in \Omega $, $\mathcal{X}$ a linear vector field and $B_1, B_2, B_3$ left-invariant vector
fields. In coordinates, $\Sigma _{\HB}$ is defined by the family of ODE's as
follows

\begin{equation*}
\Sigma _{\HB}:\quad \left\{ 
\begin{array}{l}
\mathbf{\dot{v}}=A\mathbf{v}+\omega_1\zeta+\omega_2\zeta+\omega_3\zeta  \\ 
\dot{z}=\langle \eta ,\mathbf{v}\rangle +z\mathrm{tr}A+\omega_1\alpha_1+\omega_2\alpha_2+\omega_3\alpha_3+\frac{1}{2}%
\left\langle \mathbf{v},\theta(\omega_1\zeta_1+\omega_2\zeta_2+\omega_3\zeta_3)\right\rangle 
\end{array}%
\right. 
\end{equation*}%
where $\omega\in \Omega $, $g=(\mathbf{v},z)\in \HB$, $B_i=(\zeta_i,\alpha_i )\in 
\mathfrak{h}$,  $A\in \mathfrak{gl}(2)$ and $\eta \in \mathbb{R}^{2}$.

\section{Invariant subgroups of $\HB$ by linear vector fields}

As it is well known, if a subgroup $L$ of $\HB$ is topologically closed then
the homogeneous space $L\setminus \HB$ admits a manifold structure in such way that the canonical projection $\HB\rightarrow L/\HB$ is a submersion. Hence, we start with stating completely in this subsection all possible
closed subgroups of $\HB$. Nonetheless, we exclude the trivial cases where $%
L=\{(0,0)\}$ and $L=\HB$ and only focus on the non trivial cases, namely, when
the subgroup $L$ of $\HB$ is (i) nontrivial discrete, (ii) 1-dimensional and (iii) 2-dimensional. Thus, we give, up to isomorphisms, 1-dimensional and 2-dimensional subalgebras. 

We state the following simple proposition without proof:

\begin{proposition}
\label{prop:1and2dimsubalgebras}Let $\{\mathbf{e}_{1},\mathbf{e}_{2}\}$
denote the canonical basis of $\mathbb{R}^{2}$. Then, up to isomorphisms, it holds that:

\begin{enumerate}
\item There is a unique 2-dimensional Lie subalgebra
which is 
\begin{equation*}
\mathfrak{l}_2=\mathrm{span}\{(\mathbf{e}_{1},0),(\mathbf{0},1)\}\text{.}
\end{equation*}

\item The only 1-dimensional Lie subalgebras are $%
\mathfrak{l}_{0}=\{\mathbf{0}\}\times \mathbb{R}$ or $\mathfrak{l}_{1}=%
\mathbb{R}\mathbf{e}_{1}\times \{0\}$.
\end{enumerate}
\end{proposition}

Since the exponential map $%
\exp :\mathfrak{h}\rightarrow \HB$ is the identity map (and hence, a global
diffeomorphism) it follows from the previous proposition that, up to isomorphisms, any subgroup $L\subset\HB$ has identity component given by $L_0=\fl_2$, if the $L$ is two-dimensional, or the dimension of $L$ is equal to 1 and  $L_0=\mathfrak{l}_{0}$ or  $L_0=\mathfrak{l}_{1}$.

The proposition below together with the Proposition \ref%
{prop:1and2dimsubalgebras} help to construct homogeneous spaces we need for
later references.

\begin{proposition}\label{MainPro1}
			Up to isomorphisms, any closed subgroup $L\subset \HB$ is given by:
			\begin{enumerate}
				\item $\dim L=2$ and $L=(\R\times \Z p)\times\R, p=0,1$;
				\item $\dim L=1$ and $L=\Z^k\times\R, \;\mbox{ for }k=0, 1, 2$ or $L=\R\mathbf{e}_1\times\Z p, \;\mbox{ for }p=0, 1$;
				\item $\dim L=0$ and $L=\Z\mathbf{e}_1\times\Z p$ for $p=0, 1$ , $L=\{\mathbf{0}\}\times \Z$ or $L=\Z^2\times\Z \frac{1}{p}$ for $p\in\N$.
			\end{enumerate}
		\end{proposition}
		
		\begin{proof}
1. Let $L\subset \HB$ be a closed subgroup with $\dim L=2$ and assume w.l.o.g. that $L_0=\R \mathbf{e}_{1}\times \R$. The projection $$\pi:\mathbb{H}\rightarrow \R, \hspace{1cm}\pi(\mathbf{v},z)=y,$$ 
where $\mathbf{v}=(x, y)$, is a group homomorphism with kernel given exactly by $\R \mathbf{e}_{1}\times \R$. Hence we obtain that $L_0 \setminus\HB = \R $ by the isomorphism Theorem. In particular, $\pi:\mathbb{H}\rightarrow\R$ takes $L$ into a discrete subgroup of $\R$ and hence 
$\pi(L)=\Z a$, for some $a\geq 0$. Therefore, 
$$L\subset \pi^{-1}(\Z a)=(\R\times\Z a)\times\R.$$
On the other hand, for any $g\in (\R\times\Z a)\times\R$ there exists $g_0\in L$ such that $\pi(g)=\pi(g_0)$. Consequently,
$$g*g_0^{-1}\in \R\mathbf{e}_{1}\times \R=L_0\hspace{.5cm}\implies\hspace{.5cm}w\in L_0 \ast g_0\subset L\hspace{.5cm}\implies\hspace{.5cm}L=(\R\times\Z a)\times\R.$$

If $a=0$ the item is proved. If $a\neq 0$, the map
$$\phi:\HB\rightarrow\HB, \hspace{1cm}\phi((x, y), z)=((x, a^{-1}y), z),$$
is an automorphism taking $L$ to $(\R\times\Z)\times\R$, concluding the prove.


2. Let us first assume that $L_{0}=\{\mathbf{0}\}\times \mathbb{R}$. In this case, the homogeneous space $L_{0}\setminus \HB$ coincides with the Lie group $\mathbb{R%
}^{2}$ with canonical projection given by 
$$\pi :\HB\rightarrow \mathbb{R}^{2}, \hspace{1cm}\pi(\mathbf{v}, z)=\mathbf{v}.$$ 

As previously, $\pi$ takes $L$
into a discrete subgroup of $\mathbb{R}^{2}$ implying that
$$\pi(L)=a\Z\times b\Z, \hspace{.5cm}\mbox{ for some } \hspace{.5cm} a, b\geq 0.$$
Thus, we conclude that $\pi (L)=\pi \left( 
a\Z\times b\Z\times \{0\}\right)=a\Z\times b\Z$ and hence,
				$$a\Z\times b\Z\times \{0\}\subset L*(\{\mathbf{0}\}\times\R)=L\;\;\;\mbox{ and }\;\;\; L\subset (a\Z\times b\Z\times \{0\})*(\{\mathbf{0}\}\times\R)=a\Z\times b\Z\times\R,$$
			   where the former equality follows from the fact that $Z(\HB)=\{\mathbf{0}%
\}\times \mathbb{R}$ is the center of $\HB$. As in the previous item, one can easily construct an isomorphism of $\HB$ taking $a\Z\times b\Z\times\R$ onto $\Z^k\times\R$, where $k$ depends on the numbers $a$ and $b$. Therefore the equality desired follows.

 Now assume that $L_{0}=\mathbb{R}\mathbf{e}_{1}\times
\{0\}$. If $g=(\mathbf{v},z)\in L$ then it follows that 
\begin{equation*}
g \ast L_{0}=\left\{ \left( \mathbf{v}+t\mathbf{e}_{1},z+\frac{t}{2}\left\langle 
\mathbf{v},\mathbf{e}_{2}\right\rangle \right) :t\in \mathbb{R}\right\} 
\end{equation*}%
is a line passing through the point $g=(\mathbf{v},z)$ and parallel to the
vector $\left( 2\mathbf{e}_{1},\left\langle \mathbf{v},\mathbf{e}%
_{2}\right\rangle \right) $. On the other hand,
\begin{align*}
& \bigg(\mathbf{v}+t\mathbf{e}_{1},z+\frac{t}{2}\left\langle \mathbf{v},%
\mathbf{e}_{2}\right\rangle \bigg)\ast \bigg(\mathbf{v}+s\mathbf{e}_{1},z+%
\frac{s}{2}\left\langle \mathbf{v},\mathbf{e}_{2}\right\rangle \bigg) \\
& =\left( 2\mathbf{v}+(t+s)\mathbf{e}_{1},2z+\frac{t+s}{2}\left\langle 
\mathbf{v},\mathbf{e}_{2}\right\rangle +\frac{1}{2}\left\langle \mathbf{v}+t%
\mathbf{e}_{1},\left( \mathbf{v}+s\mathbf{e}_{1}\right) ^{\ast
}\right\rangle \right)  \\
& =\left( 2\mathbf{v}+(t+s)\mathbf{e}_{1},2z+\frac{t+s}{2}\left\langle 
\mathbf{v},\mathbf{e}_{2}\right\rangle +\frac{1}{2}\left( s\left\langle 
\mathbf{v},\mathbf{e}_{2}\right\rangle -t\left\langle \mathbf{v},\mathbf{e}%
_{2}\right\rangle \right) \right)  \\
& =\big(2\mathbf{v}+(t+s)\mathbf{e}_{1},2z+s\left\langle \mathbf{v},\mathbf{e%
}_{2}\right\rangle \big) \\
& =2(\mathbf{v},z)+t\left( \mathbf{e}_{1},0\right) +s\left( \mathbf{e}%
_{1},\left\langle \mathbf{v},\mathbf{e}_{2}\right\rangle \right) 
\end{align*}%
which shows that 
\begin{equation*}
\left( g \ast L_{0}\right) ^{2}=\left\{ 2(\mathbf{v},z)+t\left( \mathbf{e}%
_{1},0\right) +s\left( \mathbf{e}_{1},\left\langle \mathbf{v},\mathbf{e}%
_{2}\right\rangle \right) :t,s\in \mathbb{R}\right\} \text{.}
\end{equation*}%
Note that $\left( g \ast L_{0}\right) ^{2}$ is a plane if $\left\langle \mathbf{v},%
\mathbf{e}_{2}\right\rangle \neq 0$. Since $L$ is a one dimensional subgroup, the fact that $(g \ast L_0)^2\subset L$ implies that 
$$(\mathbf{v},z)\in L \hspace{.5cm}\iff \hspace{.5cm}\left\langle \mathbf{v},\mathbf{e}%
_{2}\right\rangle =0 \hspace{.5cm}\iff \hspace{.5cm}\mathbf{v}\in \mathbb{R}%
\mathbf{e}_{1},$$
showing that $L\subset \mathbb{R}\mathbf{e}_{1}\times \mathbb{%
R}$. On the other hand, since any $(\mathbf{v},z)\in \HB$ can be written as $(\mathbf{v}%
,z)=(\mathbf{v},0)\ast (\mathbf{0},z)$, then 
$$(\mathbf{v},z)\in L \hspace{.5cm}\implies \hspace{.5cm} (\mathbf{0},z)=(\mathbf{v},0)^{-1}*(\mathbf{v},z)\in L\hspace{.5cm}\implies \hspace{.5cm}(\mathbf{0},z)\in (\{\mathbf{0}\}\times 
\mathbb{R)}\cap L,$$
and this shows that 
\begin{equation}
L=\mathbb{R}\mathbf{e}_{1}\times \{0\}\ast ((\{\mathbf{0}\}\times \mathbb{R}%
)\cap L)\text{.}  \label{eq:subgroup 0 times R intsec L}
\end{equation}%
However, $(\{\mathbf{0}\}\times \mathbb{R})\cap L$ in (\ref{eq:subgroup 0
times R intsec L}) is a discrete subgroup of the Lie group $\{\mathbf{0}%
\}\times \mathbb{R}$ and hence it happens that 
\begin{equation*}
(\{\mathbf{0}\}\times \mathbb{R})\cap L=\{\mathbf{0}\}\times \mathbb{Z} a, \mbox{ for some } a\geq 0.
\end{equation*}
Again, by the fact that $\{\mathbf{0}\}\times \mathbb{R}$ is
the center of $\HB$ we conclude that 
\begin{equation*}
L=\mathbb{R}\mathbf{e}_{1}\times \{0\}\ast \left( \{\mathbf{0}\}\times \mathbb{Z} a
\right) =\mathbb{R}\mathbf{e}_{1}\times \mathbb{Z} a\text{.}
\end{equation*}	
If $a=0$ we get that $L=\R\mathbf{e}_1\times\{0\}$ and for $a\neq 0$, the isomorphism 
$$(\mathbf{v}, z)\mapsto \left(\frac{1}{\sqrt{a}}\mathbf{v}, \frac{1}{a}z\right),$$
takes $L$ to the subgroup $\R\mathbf{e}_1\times\Z$ as stated.

3. As in the preceding case, by considering the subgroup $\{\mathbf{0}\}\times \mathbb{R}$ we have that $\pi (L)$ is a discrete subgroup of $(\{\mathbf{0}\}\times \mathbb{R})\setminus \HB= \R^2$ and hence we obtain that 
$$L\subset a\Z\times b\Z\times \mathbb{R}, \mbox{ for some }a,b\geq 0.$$ 
Moreover,
$$(\mathbf{v}, z)=(\mathbf{v},0)\ast (\mathbf{0},z)\hspace{.5cm}\implies\hspace{.5cm}(\mathbf{0}, z)\in (\{\mathbf{0}\}\times\R)\cap L,$$ 
which, as previously, allow us to conclude that 
$$L=a\Z\times b\Z\times c\Z, \mbox{ where }a, b, c\geq  0.$$
If $a=b=0$ then $c>0$ and the automorphism
$$(\mathbf{v}, z)\mapsto \left(\frac{1}{\sqrt{c}}\mathbf{v}, \frac{1}{c}z\right),$$
takes $L$ to the subgroup $\{\mathbf{0}\}\times\Z$. Analogously, if $a=0$ and $b\neq 0$ or $b=0$ and $a\neq 0$ and $L$ is isomorphic to $\Z\mathbf{e}_1\times\Z p$ for $p=0, 1$. On the other hand, if $ab\neq 0$, then
$$(a\mathbf{e}_1, 0), (b\mathbf{e}_2, 0)\in L\hspace{.5cm}\implies\hspace{.5cm} \left(a\mathbf{e}_1+b\mathbf{e}_2, -\frac{ab}{2}\right)\in L\hspace{.5cm}\implies\hspace{.5cm} ab\in c\Z,$$
and hence, $L$ is isomorphic to $\Z^2\times \Z \frac{1}{p}$ by the isomorphism
$$((x, y), z)\mapsto\left(\left(\frac{1}{a}x, \frac{1}{b}y\right), \frac{1}{ab}z\right), \hspace{.5cm}\mbox{ where }p=\frac{ab}{c}$$ 
for some $p\in\N$, concluding the proof.
\end{proof}

\begin{remark}
\label{NormalS} An elemantary calculation shows that 
\begin{equation*}
\left( \mathbf{v}_{1},z_{1}\right) \ast \left( \mathbf{v}_{2},z_{2}\right)
\ast \left( \mathbf{v}_{1},z_{1}\right) ^{-1}=\left( \mathbf{v}%
_{2},z_{2}+\left\langle \mathbf{v}_{1},\theta\mathbf{v}_{2}\right\rangle
\right) 
\end{equation*}%
and hence, up to isomorphisms, the only normal subgroups of $\HB$ are

(i) $(\mathbb{R}\times \mathbb{Z}p)\times \mathbb{R}$ for $p=0,1$

(ii) $\mathbb{Z}^{k}\times \mathbb{R}$ for $k=0,1,2$ and

(iii) $\{\mathbf{0}\}\times \mathbb{Z}$.
\end{remark}
Following \cite[Proposition 4]{Jouan}, if $L\subset \HB$ is a closed subgrouop, then a linear vector field $\XC$ is conjugated to a vector field on the homogeneous space $L\setminus\HB$ if and only if $L$ is invariant by the flow of $\XC$. Therefore, our next step is to obtain conditions for a 1-parameter subgroup of automorphisms $\left\{ \varphi _{t}\right\}
_{t\in \mathbb{R}}\subset \mathrm{Aut}(\HB)$ to let a closed sugbroup of $\HB$ invariant. 

Moreover, we will not take into account the normal subgroups of $\HB$ mentioned in Remark \ref{NormalS} since otherwise the
corresponding homogeneous spaces become Lie groups and the LCSs on such
spaces has been already studied in a series of papers. See, \cite{Ayala}, 
\cite{VA1},\cite{VA2},\cite{VA3},\cite{VAP} for detailed exposition.

\begin{proposition}\label{MainPro}
    	Let $\XC=(\eta, A)$ be a linear vector field on $\HB$ with associated flow $\{\varphi_t\}_{t\in\R}$. It holds:
    	\begin{enumerate}
    		\item $\Z^2\times\Z \frac{1}{p}, p\in\N$ is $\varphi_t$-invariant if and only if $\DC\equiv 0$;
    		
    		\item $\Z\cdot \mathbf{e}_{1}\times \{0\}$ is $\varphi_t$-invariant if and only if 
    		$$A\mathbf{e}_{1}=0, \;\;A\mathbf{e}_{2}=\beta \mathbf{e}_{2}+\alpha \mathbf{e}_{1}\;\;\mbox{ and }\;\;\eta\in\R \mathbf{e}_{2}, \;\mbox{ with }\;\alpha=0 \;\mbox{ if }\;\eta\neq 0;$$ 
    		
    		\item $\Z\cdot \mathbf{e}_{1}\times \Z$ is $\varphi_t$-invariant if and only 
    		$$A \mathbf{e}_{1}=0, \;\;A\mathbf{e}_{2}=\alpha \mathbf{e}_{1}\;\;\mbox{ and }\;\;\eta\in\R \mathbf{e}_{2}, \;\mbox{ with }\;\alpha=0 \;\mbox{ if }\;\eta\neq 0;$$
    		\item $\R\cdot \mathbf{e}_{1}\times\{0\},$ is $\varphi_t$-invariant if and only if 
    		$$A \mathbf{e}_{1}=\lambda \mathbf{e}_{1}, \;\;A\mathbf{e}_{2}=\beta \mathbf{e}_{2}+\alpha \mathbf{e}_{1}\;\;\mbox{ and }\;\;\eta\in\R \mathbf{e}_{2}, \;\mbox{ with }\;\alpha=0 \;\mbox{ if }\;\eta\neq 0;$$
    		
    		\item $\R\cdot \mathbf{e}_{1}\times\Z$ is $\varphi_t$-invariant if and only if 
    		$$A \mathbf{e}_{1}=\lambda \mathbf{e}_{1}, \;\;A\mathbf{e}_{2}=-\lambda \mathbf{e}_{2}+\alpha \mathbf{e}_{1}\;\;\mbox{ and }\;\;\eta\in\R \mathbf{e}_{2}, \;\mbox{ with }\;\alpha=0 \;\mbox{ if }\;\eta\neq 0;$$ 
    		
\end{enumerate}

    \end{proposition}
  \begin{proof}
1.~ Pick a point $(\mathbf{v},z)\in \mathbb{Z}^{2}\times \mathbb{Z}\frac{1}{p},$ where $p\in \N
$ and assume that 
\begin{equation*}
\varphi _{t}(\mathbf{v},z)=\bigg(e^{tA}\mathbf{v}~,~\langle e^{t\cdot\mathrm{tr}A}%
\mathbf{\Lambda }_{t}^{A-\mathrm{tr}A\cdot I_2}\eta ,\mathbf{v}\rangle +ze^{t\cdot\mathrm{%
tr}A}\bigg)\in \mathbb{Z}^{2}\times \mathbb{Z}\frac{1}{p}
\end{equation*}%
Then the following equations are obtained 
\begin{equation}
\begin{cases}
e^{tA}\mathbf{v}\in \mathbb{Z}^{2} \\ 
\langle e^{t\cdot\mathrm{tr}A}\mathbf{\Lambda }_{t}^{A-\mathrm{tr}A\cdot I_2}\eta ,\mathbf{%
v}\rangle +ze^{t\cdot\mathrm{tr}A}\in \mathbb{Z}\frac{1}{p}\text{.}%
\end{cases}
\label{Eq:mainpro1}
\end{equation}%
It results from the first equation of (\ref{Eq:mainpro1}) $A\equiv 0$ and $\eta $ is
orthogonal to the any vector $\mathbf{v}\in\Z^2$, implying that $\eta =0$. Thus, $\DC=0$.\newline
\vspace{3mm} 2. Now, $\mathbb{Z}\mathbf{e}_1\times
\{0\}$ is $\varphi_t$ invariant if and only if it holds that
\begin{equation}
\begin{cases}
e^{tA}\mathbf{e}_1\in \mathbb{Z}\mathbf{e}_1 \\ 
\langle \mathbf{\Lambda }_{t}^{A-\mathrm{tr}A\cdot I_2}\eta ,\mathbf{%
e}_1\rangle=0.%
\end{cases}
\label{Eq:mainpro2}
\end{equation}%
The first equation of (\ref{Eq:mainpro2}) implies that $t\mapsto e^{tA}\mathbf{e}_{1}$ must be constant since it is
a  continuous curve and $\mathbb{Z}$ is a discrete subgroup. Now, if we take
its derivative at $t=0$ we immediately get $A\mathbf{e}_{1}=0$. 

On the other hand, from the second equation, we get that
$$0=\frac{d}{dt}_{|t=0}\langle \mathbf{\Lambda }_{t}^{A-\mathrm{tr}A\cdot I_2}\eta ,\mathbf{%
e}_1\rangle=\langle \rme^{t\cdot(A-\mathrm{tr}A\cdot I_2)}\eta ,\mathbf{%
e}_1\rangle|_{t=0}=\langle \eta, \mathbf{e}_1\rangle,$$
from which we conclude that the matrix $A$ satisfies 
$$A\mathbf{e}%
_{1}=0,\hspace{.5cm}A\mathbf{e}_{2}=\beta \mathbf{e}_{2}+\alpha \mathbf{e}_{1}\hspace{.5cm}\text{ and 
}\hspace{.5cm}\eta \in \mathbb{R}\mathbf{e}_{2}~\text{ with }~\alpha =0~\text{ if }~\eta
\neq 0,$$
as stated.

\vspace{3mm} 3. As previously, $\mathbb{Z}\mathbf{e}_1 \times 
\mathbb{Z}$ is $\varphi_t$-invariant if and only if, for all $n, m\in\Z$, 
\begin{equation}
\begin{cases}
e^{tA}\mathbf{e}_1\in \mathbb{Z}\mathbf{e}_1 \\ 
n\langle e^{t\cdot\mathrm{tr}A}\mathbf{\Lambda }_{t}^{A-\mathrm{tr}A\cdot I_2}\eta ,\mathbf{%
e}_1\rangle +m\rme^{t\cdot\mathrm{tr}A}\in \mathbb{Z}\text{.}%
\end{cases}
\label{eq:mainpro25}
\end{equation}%
If we choose $n=0$ and $m=1$, we get from the second equation in (\ref{eq:mainpro25}) that $e^{t\cdot \mathrm{%
tr}A}\in \mathbb{Z}$ for all $t\in \mathbb{R}$ which results $\mathrm{tr}A=0$. Similarly, if we select $n=1$ and $m=0$ then%
\begin{equation*}
\rme^{tA}\mathbf{e}_{1}\in \Z\mathbf{e}_1 \hspace{.5cm}\mbox{ and }\hspace{.5cm}\langle \mathbf{\Lambda }%
_{t}^{A}\eta ,\mathbf{e}_{1}\rangle \in \mathbb{Z},
\end{equation*}%
from which we get, like in the preceding item, that $A\mathbf{e}_1=0$ and $\eta \in \mathbb{R}\mathbf{e}_{2}$. Since $\tr A=0$ we get already that 
$$A\mathbf{e}_{1}=\mathbf{0}\hspace{.5cm}A\mathbf{e}_{2}=\alpha \mathbf{e}_{1}\hspace{.5cm}\mbox{ and }\hspace{.5cm}\eta\in\R\mathbf{e}_2.$$
Now, 
$$0=\frac{d^2}{dt^2}_{|t=0}\langle \mathbf{\Lambda }%
_{t}^{A}\eta ,\mathbf{e}_{1}\rangle=\langle A\rme^{tA}\eta, \mathbf{e}_1\rangle |_{t=0}=\langle A\eta, \mathbf{e}_{1}\rangle\hspace{.5cm}\implies\hspace{.5cm} A\mathbf{e}_{2}=0, \;\;\mbox{ if }\;\;\eta\neq 0,$$
concluding the proof. 
\vspace{3mm}

4. Analogously as the previous cases, $\R\mathbf{e}_1\times\{0\}$ is $\varphi_t$-invariant if and only if  
\begin{equation}
\begin{cases}
e^{tA}\mathbf{e}_1\in \mathbb{R}\mathbf{e}_1 \\ 
\langle \mathbf{\Lambda }_{t}^{A-\mathrm{tr}A\cdot I_2}\eta ,\mathbf{%
e}_1\rangle= 0,
\end{cases}
\label{Eq:mainpro4}
\end{equation}%
which gives us $A\mathbf{e}_{1}=\lambda \mathbf{e}_{1}$ and, by derivation of the second equation in (\ref{Eq:mainpro4}) at $t=0$, $\left\langle \eta,%
\mathbf{v}\right\rangle =0$. Consequently, if we
write the matrix $A$ in canonical form, we see that $A$ is such that  
$$A\mathbf{e}_{1}=\lambda \mathbf{e}_{1},\hspace{.5cm} A\mathbf{e}_{2}=\beta \mathbf{e}_{2}+\alpha \mathbf{e}_{1}\hspace{.5cm}\text{ and }\hspace{.5cm}\eta \in \mathbb{R}\mathbf{e}_{2}~%
\text{with}~\alpha =0~\text{if}~\eta \neq 0,$$
showing the assertion.
\vspace{3mm}

5. The proof is similar as the items 3. and 4. and we will omit it.
\end{proof}

\section{Linear control systems on homogeneous spaces of $\HB$}

In this section, we classify all possible LCSs on homogeneous spaces
of $\HB$. For it, let $L\subset \HB$ be a closed subgroup and denote by $\pi:\HB\rightarrow L\setminus \HB$ the standard canonical projection.

\begin{definition}
A LCS on the homogeneous space $L\setminus \HB$ is the following control-affine system : 
\begin{equation}
\Sigma _{L\setminus \HB}:\quad \dot{P}=f_{0}(P)+\sum_{j=1}^m u_jf_{j}(P)
\label{Eq:contsyshomogeneous}
\end{equation}%
with $u\in \Omega ,P\in L\setminus\HB$ and $f_{0},f_{1}, \ldots, f_m$ are vector fields on $L\setminus \HB$ satisfying 
\begin{equation*}
\pi _{\ast }\circ \mathcal{X}=f_{0}\circ\pi \hspace{.5cm}\text{ and }\hspace{.5cm}\pi _{\ast }\circ B_j=f_{j}\circ\pi,\;\;j=1, \ldots, m,
\end{equation*}
\end{definition}
where $\XC$ is a linear vector field and $B_j$ are left-invariant vector fields and $m=3-\dim L$.

It follows from the Definition (\ref{defi:conjugado}) that a LCS on a
homogeneous space $L\setminus \HB$ is $\pi $-conjugated to a LCS on $\HB$. We also know (See Proposition 4.1., \cite{Jouan}) that the vector field $\pi_{\ast }\circ \mathcal{X}$ is well defined on $L\setminus \HB$ if and only if $L$ is invariant by the flow $\varphi _{t}$ of $\mathcal{X}$ which means $\varphi _{t}(L)=L$ for every $t\in \mathbb{R}$.

Now it becomes clear that in order to classify all the possible LCSs on the homogeneous spaces of $\HB$ we need to find the possible $\varphi_{t}$-invariant closed subgroups of $\HB$. Let $\Sigma_{L\setminus \HB}$ denote a LCS on $L\setminus \HB$ as in (\ref{Eq:contsyshomogeneous}) such that $\mathcal{X}$ and $B_j$, $j=1, \ldots, m$ are $\pi $-conjugated with the vector fields $f_{0}$ and $f_{1}, \ldots, f_m$, respectively. Let $\psi \in \mathrm{Aut}(\HB)$ such that $\widehat{L}=\psi (L)$ is one of the subgroups in Proposition \ref{MainPro}. Consider $\widehat{\mathcal{X}}$ and $\hat{B}_j$, $j=1, \ldots, m$ the vector fields satisfying
\begin{equation*}
\psi _{\ast }\circ \widehat{\mathcal{X}}=\mathcal{X}\circ \psi \hspace{.5cm}\mbox{ and }\hspace{.5cm} \psi _{\ast }\circ \widehat{B}_j=B_j\circ \psi, \;\;j=1, \ldots, m.
\end{equation*}%
That $L$ is invariant under the flow of $\mathcal{X}$ implies that $\widehat{L}$ is also invariant under the flow of $\widehat{\mathcal{X}}$. Hence we have well defined vector fields $\widehat{f}_{0}$ and $\widehat{f}_{1}, \ldots, \widehat{f}_{m}$ on $%
\widehat{L}\setminus \HB$ determined by the relations 
\begin{equation*}
\widehat{f_{0}}\circ \widehat{\pi }=\widehat{\pi }_{\ast }\circ \widehat{\mathcal{X}}\hspace{.5cm}\mbox{ and }\hspace{.5cm} 
\widehat{f_{1}}\circ \widehat{\pi }=\widehat{\pi }_{\ast }\circ \widehat{B},  \;\;j=1, \ldots, m,
\end{equation*}%
where $\widehat{\pi}:\HB\rightarrow \widehat{L}\setminus \HB$ is the canonical projection. Since the map $\widehat{\psi }:L\setminus \HB\rightarrow  \widehat{L}\setminus \HB$ defined by the relation $\widehat{\psi }\circ \pi =\widehat{\pi }\circ \psi $ is a diffeomorphism, the fact that 
\begin{equation*}
\widehat{\psi }_{\ast }\circ f_{0}=\widehat{f}_{0}\circ \widehat{\psi }\hspace{.5cm}\mbox{ and }\hspace{.5cm} \widehat{\psi}_{\ast }\circ f_{j} =\widehat{f}_{j}\circ \widehat{\psi }, \;\;j=1, \ldots m,
\end{equation*}
shows us that $\Sigma _{L\setminus \HB}$ is equivalent to the LCS on $\widehat{L}\setminus \HB$ given by
\begin{equation*}
\Sigma _{\widehat{L}\setminus \HB}:\quad \dot{Q}=\widehat{f}_{0}(Q)+\sum_{j=1}^mu_j\widehat{
f}_{j}(Q).
\end{equation*}%
As a result, we can
assume that subgroup $L$ is one of the subgroups determined in Proposition \ref{MainPro}, and we just need to examine the following cases.

\subsection{The zero-dimensional (i.e., discrete) case}
In this section we consider the homogeneous spaces of $\HB$ by zero-dimensional subgroups. By Proposition \ref{MainPro1} and Remark \ref{NormalS}, up to isomorphisms, the only subgroups we have to consider are 
   $$\Z^2\times\Z \frac{1}{p}, \;\; p\in\N\;\;\;\mbox{ and }\;\;\;\Z\mathbf{e}_1\times \Z p, \;\;p=0, 1.$$
   However, by Proposition \ref{MainPro} the subgroup $\Z^2\times\Z \frac{1}{p}, p\in\N$ is invariant by the flow of a linear vector field $\XC$ if and only if $\XC\equiv 0$. As a consequence, any induced LCS on the homogeneous space $\left(\Z^2\times\Z \frac{1}{p}\right)\setminus\HB$ is driftless left-invariant system and hence its dynamical behaviour is well known (See for instance \cite{Sussmann}). 
   
   Let us consider the case $L=\Z\mathbf{e}_1\times\Z p$ for $p=0,1$. If $(\mathbf{v}_{1}, z_1), (\mathbf{v}_{2}, z_2)\in\HB$ are such that  $L*(\mathbf{v}_{1}, z_1)=L*(\mathbf{v}_{2}, z_2)$, then by the definition there exists $(m, n)\in\Z\times\Z p$  such that 
   $$(m\mathbf{e}_{1}, n)*(\mathbf{v}_{1}, z_1)=(\mathbf{v}_{2}, z_2)\iff \left(m\mathbf{e}_{1}+\mathbf{v}_{1}, z_1+\frac{1}{2}\langle m\mathbf{e}_{1}, \theta\mathbf{v}_{1}\rangle+n\right)=(\mathbf{v}_{2}, z_2)\iff \left\{\begin{array}{l}
   	m\mathbf{e}_{1}=\mathbf{v}_{2}-\mathbf{v}_{1}\\ z_2=z_1+\frac{m}{2}\langle \mathbf{e}_{1}, \theta\mathbf{v}_{1}\rangle+n
   \end{array}\right.$$
   Using the first equation we obtain that 
			 $$m=\langle \mathbf{v}_{2}, \mathbf{e}_{1}\rangle-\langle \mathbf{v}_{1}, \mathbf{e}_{1}\rangle, \hspace{.5cm}\langle \mathbf{v}_{2}, \mathbf{e}_{2}\rangle=\langle \mathbf{v}_{1}, \mathbf{e}_{2}\rangle, \hspace{.5cm}\mbox{ and }\hspace{.5cm}\langle \mathbf{e}_{1}, \theta\mathbf{v}_{1}\rangle=\langle \mathbf{e}_{1}, \theta\mathbf{v}_{2}\rangle.$$
			 
			 Using now the second equation and the previous relations, gives us that 
			 $$z_2=z_1+\frac{1}{2}\langle t\mathbf{e}_{1}, \theta\mathbf{v}_{1}\rangle+n=z_1+\frac{1}{2}(\langle \mathbf{v}_{2}, \mathbf{e}_{1}\rangle-\langle \mathbf{v}_{1}, \mathbf{e}_{1}\rangle)\langle \mathbf{e}_{1}, \theta\mathbf{v}_{1}\rangle+n$$
			 $$=z_1-\frac{1}{2}\langle \mathbf{v}_{1}, \mathbf{e}_{1}\rangle\langle \mathbf{e}_{1}, \theta\mathbf{v}_{1}\rangle+\frac{1}{2}\langle \mathbf{v}_{2}, \mathbf{e}_{1}\rangle\langle \mathbf{e}_{1}, \theta\mathbf{v}_{2}\rangle+n$$
			 and so 
			 $$\left(z_2+\frac{1}{2}\langle \mathbf{v}_{2}, \mathbf{e}_{1}\rangle\langle \mathbf{v}_{2}, \mathbf{e}_{1}\rangle\right)-\left(z_1+\frac{1}{2}\langle \mathbf{v}_{1}, \mathbf{e}_{1}\rangle\langle \mathbf{v}_{1}, \mathbf{e}_{2}\rangle\right)=n.$$
   
   Therefore, $L*(\mathbf{v}_{1}, z_1)=L*(\mathbf{v}_{2}, z_2)$ if and only if
   \begin{equation}
   	[x_1]_1=[x_2]_1, \;\;\; y_1=y_2\;\;\mbox{ and }\;\;\left[z_1+\frac{1}{2}x_1y_1\right]_p=\left[z_2+\frac{1}{2}x_2y_2\right]_p
   \end{equation}
   where $\mathbf{v}_{i}=(x_i, y_i), i=1, 2$, $[x]_1=x+\Z$ and $[x]_0=x$.
   
   Therefore, the homogeneous space $((\Z\mathbf{e}_1\times \Z p)\setminus \HB$ is identified with $(\T\times\R)\times \T^p$, where $\T^0:=\R$ and $\T^1=\R/\Z$. The canonical projection is given by 			 
   $$\pi_{0, p}:\HB\rightarrow (\T\times\R)\times\T^p, \;\;\;((x, y), z)\mapsto \left([x], y, \left[z+\frac{1}{2}xy\right]_p\right).$$
   
    \begin{remark}
    	\label{map2}
    	By considering the maps $f:\HB\rightarrow\HB$ and $h_p:\HB\rightarrow \R\times\T^p$ defined, respectively, as
    	$$f((x, y), z)=\left(x, y, z+\frac{1}{2}xy\right)\;\;\;\mbox{ and }\;\;\; h_p((x, y), z)=([x], y, [z]_p),$$
    	and using that the differential of the canonical projection $\R\rightarrow \R/\Z$ is the identity map, it is easy to see that 
    	$$h_p\circ f=\pi_{0, p}\;\;\;\;\mbox{ and }\;\;\;\;(\pi_{0, p})_*=f_*, \mbox{ for }p=0, 1.$$
    \end{remark}

\subsubsection{LCS's on $(\Z\mathbf{e}_1\times \Z p)\setminus\HB$, $p=0, 1$} 

As previously, by Proposition \ref{MainPro}, if $\XC=(\eta, A)$ is a linear vector field on $\HB$ whose flow let $\Z\mathbf{e}_1\times \Z p$, $p=0, 1$ invariant then 
$$\eta=(0, \gamma)\;\;\;\mbox{ and }\;\;\;A=\left(\begin{array}{cc}
	0 & \alpha\\ 0 & p\beta
\end{array}\right), \;\;\mbox{ with }\alpha=0\;\mbox{ if }\;\gamma\neq 0.$$
Therefore, in coordinates, $\XC$ is given by the expression 
$$\XC((x, y), z)=((\alpha y, p\beta y), \gamma y+p\beta z).$$
Using Remark \ref{map2}, we get that 
$$(d\pi_{0, p})_{((x, y), z)}=\left(\begin{array}{ccc}
1 & 0 & 0 \\0 & 1 & 0\\ \frac{1}{2}y & \frac{1}{2}x & 1
\end{array}\right),$$
implying that
$$(d\pi_{0, 0})_{((x, y), z)}\XC((x, y), z)=\left(\begin{array}{ccc}
1 & 0 & 0\\	0 & 1 & 0\\ \frac{1}{2}y & \frac{1}{2}x & 1
\end{array}\right)\left(\begin{array}{c}
	\alpha y\\p\beta y\\\gamma y+p\beta z
\end{array}\right)$$
$$=\left(\begin{array}{c}
\alpha y\\	p\beta y\\ p\beta\left(z+\frac{1}{2}xy\right)+\frac{1}{2}\alpha y^2+2\gamma y
\end{array}\right)=\left(\alpha y, p\beta y, p\beta\left(z+\frac{1}{2}xy\right)+\frac{1}{2}\alpha y^2+2\gamma y\right),$$
and hence,
$$\widehat{\XC}_{0, p}([u]_1, s, [t]_p)=\left(\alpha s, p\beta s, p\beta t+\frac{1}{2}\alpha s^2+2\gamma s\right), \;\;\mbox{ with }\;\;\alpha=0\;\mbox{ if }\;\;\gamma\neq 0,$$
is the general expression of a vector field on $(\Z\mathbf{e}_1\times \Z p)\setminus\HB$ induced by a linear vector field on $\HB$.

Now, let us consider a left-invariant vector field $B$. In coordinates, we have that
$$B((x, y), z)=\left((a, b), c+\frac{1}{2}(ay-bx)\right),$$
and hence, 
\begin{align*}
\left( d\pi _{0,p}\right) _{((x,y),z)}B((x,y),z)=\left( 
\begin{array}{lll}
1 & 0 & 0\\
0 & 1 & 0 \\ 
\frac{1}{2}y & \frac{1}{2}x & 1%
\end{array}%
\right) & \left( 
\begin{array}{c}
a \\ 
b \\ 
c+\frac{1}{2}(ay-bx)%
\end{array}%
\right)  
=\left( 
\begin{array}{c}
a\\
b \\ 
c+ay%
\end{array}%
\right) & =\big(a, b,c+ay)\text{.}
\end{align*}

Therefore, the general expression of a vector field on $(\Z\mathbf{e}_1\times \Z p)\setminus\HB$ induced by a left-invariant vector field on $\HB$ is given by 
$$\widehat{B}_{0, p}([u]_1, s, [t]_p)=\left(a, b, c+as\right), \;\;([u]_1, s, [t]_p)\in\T\times\R\times\T^p.$$
As a consequence, we have the following expression for a general LCS on $(\T\times\R)\times \T^p$ for $p=0,1$.

\begin{proposition}
	A LCS on $(\Z\mathbf{e}_1\times \Z)\setminus\HB\simeq \T\times\R\times\T^p, p=0, 1$ has the form
\begin{equation*}
(\Sigma _{0,p}):\quad \left\{
		\begin{array}{l}
			\dot{u}=\alpha s+\omega_1 a_1+\omega_2a_2+\omega_3a_3\\
			\dot{s}=p\beta s+\omega_1 b_1+\omega_2b_2+\omega_3b_3\\
			\dot{t}=p\beta t+\frac{1}{2}\alpha s^2+\gamma s+\omega_1c_1+\omega_2c_2+\omega_3c_3+(\omega_1 a_1+\omega_2a_2+\omega_3a_3)s
		\end{array}\right.
\end{equation*}
where $\omega \in \Omega $, $a_i, b_i, c_i, \alpha, \gamma, \lambda\in\R, i=1, 2, 3$ and $\alpha=0$ if $\gamma\neq 0$.
\end{proposition}

\subsection{The one-dimensional case}
In this section we analyze the homogeneous spaces of $\HB$ by one-dimensional subgroups. Since we are interested in the case where the homogeneous space is not a Lie group, we have by Proposition \ref{MainPro1} and Remark \ref{NormalS} that, up to isomorphisms, the only subgroups we have to consider are $\mathbb{R}\mathbf{e}_1\times \mathbb{Z} p$, $p=0, 1$. Let $\left( \mathbf{v}_{1},z_{1}\right) ,\left( \mathbf{v}_{2},z_{2}\right) \in \HB$ and assume that $L\ast \left( \mathbf{v}_{1},z_{1}\right) =L\ast \left( \mathbf{v}_{2},z_{2}\right) $, where $L=\mathbb{R}\mathbf{e}_1\times \mathbb{Z}p$ for $p=0,1$. By definition, there exists $(t,n)\in \mathbb{R}\times \mathbb{Z}$ such that 
			 $$(t\mathbf{e}_{1}, n)*(\mathbf{v}_{1}, z_1)=(\mathbf{v}_{2}, z_2)\iff \left(t\mathbf{e}_{1}+\mathbf{v}_{1}, z_1+\frac{1}{2}\langle t\mathbf{e}_{1}, \mathbf{v}_{1}^*\rangle+n\right)=(\mathbf{v}_{2}, z_2)\iff \left\{\begin{array}{l}
			 t\mathbf{e}_{1}=\mathbf{v}_{2}-\mathbf{v}_{1}\\ z_2=z_1+\frac{1}{2}\langle t\mathbf{e}_{1}, \mathbf{v}_{1}^*\rangle+n
			 \end{array}\right.$$
			 Similar calculations as in the previous case, allows us to obtain that 
			 \begin{equation*}
			 L*((x_1, y_1), z_1)=L*((x_2, y_2), z_2)\;\;\iff\;\; y_1=y_2\;\;\mbox{ and }\;\;\left[z_1+\frac{1}{2}x_1y_1\right]_p=\left[z_2+\frac{1}{2}x_2y_2\right]_p,
			 \end{equation*}
			 where by definition $[x]_1=x+\Z$ and $[x]_0=x$. Therefore, the homogeneous space $(\R\mathbf{e}_1\times \Z p)\setminus \HB$ is identified with $\R\times \T^p$, where $\T^0:=\R$ and $\T^1=\R/\Z$. The canonical projection is given by 			 
			 $$\pi_{1, p}:\HB\rightarrow \R\times\T^p, \;\;\;((x, y), z)\mapsto \left(y, \left[z+\frac{1}{2}xy\right]_p\right).$$	
			 
			 \begin{remark}
			 	\label{map}
			 Similarly as in the zero dimensional case, we can consider the maps $f:\HB\rightarrow\HB$ and $g_p:\HB\rightarrow \R\times\T^p$ given by as 
			 $$f((x, y), z)=\left(x, y, z+\frac{1}{2}xy\right)\;\;\;\mbox{ and }\;\;\; g_p((x, y), z)=(y, [z]_p),$$
			 it is easy to see that $g_p\circ f=\pi_{1, p}$ for $p=0, 1$. Moreover, since the differential of the canonical projection $\R\rightarrow \R/\Z$ is the identity, we get that $(\pi_{1, p})_*=\pi_2\circ f_*$, $p=0, 1$, where 
			 $$\pi_2:\R^3\rightarrow\R^2, \hspace{1cm}\pi_2(x, y, z)=(y, z).$$
			 
			\end{remark}
			 
		 \subsubsection{LCS's on $(\R\mathbf{e}_1\times \{0\})\setminus\HB$}
			 		
Now, let $\mathcal{X}$ be a linear vector field on $\HB$ and assume that $\mathbb{R}%
\mathbf{e}_1\times \{0\}$ is invariant under the flow of $\mathcal{X}=(\eta, A)$. By Proposition \ref{MainPro} it follows that
			 $$\eta=(0, \gamma)\;\;\;\mbox{ and }\;\;\;A=\left(\begin{array}{cc}
			 	\lambda & \alpha\\ 0 & \beta
			 \end{array}\right), \;\;\mbox{ with }\alpha=0\;\mbox{ if }\;\gamma\neq 0.$$
		     As a consequence, in coordinates, we get that
		     $$\XC((x, y), z)=((\lambda x+\alpha y, \beta y), \gamma y+(\lambda+\beta)z).$$
		     Moreover, by a simple calculation, we get that 
		     $$(df)_{((x, y), z)}=\left(\begin{array}{ccc}
		     	1 & 0 & 0\\0 & 1 & 0\\ \frac{1}{2}y & \frac{1}{2}x & 1
		     \end{array}\right)\hspace{.5cm}\implies\hspace{.5cm}(d\pi_{1, p})_{((x, y), z)}=\left(\begin{array}{ccc}
		     	0 & 1 & 0\\ \frac{1}{2}y & \frac{1}{2}x & 1
		     \end{array}\right),$$
	     and consequently, using Remark \ref{map}
\begin{equation*}
\begin{gathered} \left(d \pi_{1,0}\right)_{((x, y), z)} \mathcal{X}\big((x,
y), z\big)=\left(\begin{array}{lll} 0 & 1 & 0 \\ \frac{1}{2}y & \frac{1}{2}x & 1
\end{array}\right)\left(\begin{array}{c} \lambda x+\alpha y \\ \beta y \\
\gamma y+(\lambda+\beta) z \end{array}\right) \\ =\left(\begin{array}{c}
\beta y \\ (\lambda+\beta)\left(z+\frac{1}{2}xy\right)+\frac{1}{2}\alpha y^{2}+ \gamma y
\end{array}\right)=\left(\beta y,(\lambda+\beta)\left(z+\frac{1}{2}xy\right)+\frac{1}{2}\alpha y^{2}+
\gamma y\right) . \end{gathered}
\end{equation*}%
Therefore, the general expression of a vector field on $(\R\mathbf{e}_1\times\{0\})\setminus\HB$ induced by a linear vector field on $\HB$ is given by 
$$\widehat{\XC}_{1, 0}(s, t)=\left(\beta s, (\lambda+\beta)t+\alpha s^2+\gamma s\right), \;\;\mbox{ with }\;\;\alpha=0\;\mbox{ if }\;\;\gamma\neq 0.$$

Analogous calculations, allow us to conclude that
	 $$\widehat{B}_{1, 0}(s, t)=\left(b, c+as\right), \;\;(s, t)\in\R\times\R.$$
	 
	 We have the following.
	 
	 \begin{proposition}
	 	A  LCS on $(\R\mathbf{e}_1\times \{0\})\setminus\HB\simeq\R^2$ has the form
\begin{equation*}
(\Sigma _{1,0}):\quad \left\{ 
	 	\begin{array}{l}
	 		\dot{s}=\beta s+\omega_1 b_1+\omega_2 b_2+\omega_3 b_3\\
	 		\dot{t}=(\lambda+\beta)t+\frac{1}{2}\alpha s^2+\gamma s+\omega_1c_1+\omega_2c_2+\omega_3c_3+(\omega_1 a_1+\omega_2a_2+\omega_3a_3)s
	 	\end{array}\right.
\end{equation*}
where $\omega \in \Omega $ with  $a_i, b_i, c_i, \alpha, \beta, \gamma, \lambda\in\R, i=1, 2, 3$ and $\alpha=0$ if $\gamma\neq 0$. 
	 \end{proposition}

\subsubsection{LCS's on $(\R\mathbf{e}_1\times \Z)\setminus\HB$} 
Let us now consider the other one-dimensional subgroup $\R\mathbf{e}_1\times \Z$. By Proposition \ref%
{MainPro}, if $\R\mathbf{e}_1\times \Z$ is
invariant by the flow of $\mathcal{X}=(\eta ,A)$, we have that 
\begin{equation*}
\eta =(0,\gamma )\quad \text{and}\quad A=\left( 
\begin{array}{cc}
\lambda  & \alpha  \\ 
0 & -\lambda 
\end{array}%
\right) 
\end{equation*}%
with$~\alpha =0~$if$~\gamma \neq 0$. As a consequence, in coordinates, we
have that 
\begin{equation*}
\mathcal{X}\big((x,y),z\big)=\big((\lambda x+\alpha y,-\lambda y),\gamma y%
\big)\text{.}
\end{equation*}%
By Remark \ref{map}, we get that 
\begin{equation*}
\begin{gathered} \left(d \pi_{1,1}\right)_{((x, y), z)} \mathcal{X}((x, y),
z)=\left(\begin{array}{lll} 0 & 1 & 0 \\ \frac{1}{2}y & \frac{1}{2}x & 1
\end{array}\right)\left(\begin{array}{c} \lambda x+\alpha y \\ -\lambda y \\
\gamma y \end{array}\right) \\ =\left(\begin{array}{c} -\lambda y \\ \frac{1}{2}\alpha
y^{2}+ \gamma y \end{array}\right)=\left(-\lambda y, \frac{1}{2}\alpha y^{2}+ \gamma
y\right), \end{gathered}
\end{equation*}%
and hence, 
$$\widehat{\XC}_{1, 1}(s, t)=\left(-\lambda s,\frac{1}{2}\alpha s^2+\gamma s\right), \;\;\mbox{ with }\;\;\alpha=0\;\mbox{ if }\;\;\gamma\neq 0,$$
is the general expression of a vector field on $(\R \mathbf{e}_{1}\times\Z)\setminus\HB$ induced by a linear vector field on $\HB$. Analogous calculations, allow us to conclude that
$$\widehat{B}_{1, 1}(s, [t])=\left(b, c+as\right), \;\;(s, [t])\in\R\times\T,$$
is the general expression of a vector field on $(\R \mathbf{e}_{1}\times\Z)\setminus\HB$ induced by a left-invariant vector field. As a consequence, we have the following expression for a general LCS on $\R\times\T$.
      
\begin{proposition}
A one-input LCS on $(\R \mathbf{e}_{1}\times\Z)
\setminus \HB\simeq \mathbb{R}\times \mathbb{T}$ has the form 
\begin{equation*}
(\Sigma _{1,1}):\quad \left\{ 
\begin{array}{l}
\dot{s}=-\lambda s+\omega_1 b_1+\omega_2 b_2+\omega_3 b_3 \\ 
\dot{[t]}=\frac{1}{2}\alpha s^{2}+\gamma s+\omega_1c_1+\omega_2c_2+\omega_3c_3+(\omega_1 a_1+\omega_2a_2+\omega_3a_3)s
\end{array}%
\right. 
\end{equation*}%
where $\omega \in \Omega $ with $a_i,b_i,c_i,\alpha ,\gamma ,\lambda \in \mathbb{R}, i=1, 2, 3
$ and $\alpha =0$ if $\gamma \neq 0$.

\end{proposition}

 \section{Controllability and control sets of $\Sigma _{1,1}$}
We have classified so far all possible LCSs on homogeneous spaces of Heisenberg group and it becomes very difficult to determine controllability issue and control sets of each one of these dynamics. Hence we will content ourselves in this last section of the article studying as an example a one-input control system on the non simply connected homogeneous space $\mathbb{R}\times \mathbb{T}$ of dimension two. A much more detailed exposition involving all possible systems will be presented in a future work. Firstly, let's prove the following lemma, which we use in the proof of the related theorem.

\begin{lemma}
\label{lemma}
 Let $(\Sigma _\mathbb{R}):~\dot{s}=-\lambda s+\omega b   $ where $b \neq 0$ and $\omega\in\Omega:=[\omega_*, \omega^*]$. Then $\Sigma _\mathbb{R}$ admits only one control set $\mathcal{C}_{\Sigma_\mathbb{R}}$ satisfying:
 \begin{align*}
\begin{cases}
&\mathcal{C}_{\Sigma _\mathbb{R}}=\frac{b}{\lambda} \Omega \quad \text{if}\quad \lambda>0\quad \text{or}, \\
&\mathcal{C}_{\Sigma_\mathbb{R}}=\text{int}~(\frac{b}{\lambda} \Omega) \quad \text{if} \quad  \lambda<0 \quad \text{or}, \\
&\mathcal{C}_{\Sigma _\mathbb{R}}=\mathbb{R} \quad \text{if} \quad \lambda=0.
\end{cases}
\end{align*}   
\end{lemma}
\begin{proof}
The solutions of $\Sigma _\mathbb{R}$ are constructed by concatenations of the curves
\begin{equation*}
    \phi(\tau,s, \omega)= \rme^{-\tau\lambda} \bigg(s-\frac{b}{\lambda}\omega\bigg)+\frac{b}{\lambda}\omega.
\end{equation*}
$\bullet$~ Let $\lambda$ be positive and assume that $b>0 $ since the case  $b<0 $ is analoguous. Take any point $s_0 \in \frac{b}{\lambda}\Omega=\left[\frac{b}{\lambda}\omega_{\ast},\frac{b}{\lambda}\omega^{\ast}\right]$, then we have
\begin{align*}
 \phi(\tau,s_0,\omega)-\frac{b}{\lambda}\omega_{\ast}=&  \rme^{-\tau\lambda} \bigg(s_0-\frac{b}{\lambda}\omega\bigg)+\frac{b}{\lambda}\omega-\frac{b}{\lambda}\omega_{\ast} \geq \rme^{-\tau\lambda} \bigg(\frac{b}{\lambda}\omega_{\ast}-\frac{b}{\lambda}\omega\bigg)+\frac{b}{\lambda}\omega-\frac{b}{\lambda}\omega_{\ast}\\
 =&\underbrace{(-\rme^{-\tau\lambda} + 1)}_{>0} \underbrace{(\frac{b}{\lambda}\omega-\frac{b}{\lambda}\omega_{\ast})}_{\geq 0} \geq 0.
 \end{align*}
With similar calculations, we obtain the following 
\begin{align*}
 \phi(\tau,s_0,\omega)-\frac{b}{\lambda}\omega^{\ast}=&  \rme^{-\tau\lambda} \bigg(s_0-\frac{b}{\lambda}\omega\bigg)+\frac{b}{\lambda}\omega-\frac{b}{\lambda}\omega^{\ast} \leq e^{-\tau\lambda} \bigg(\frac{b}{\lambda}\omega{^\ast}-\frac{b}{\lambda}\omega\bigg)+\frac{b}{\lambda}\omega-\frac{b}{\lambda}w^{\ast}\\
 =&\underbrace{(\rme^{-\tau\lambda} - 1)}_{\geq 0} \underbrace{(\frac{b}{\lambda}\omega^{\ast} -\frac{b}{\lambda}\omega)}_{< 0} \leq 0.
 \end{align*}
Therefore, we find $\phi(\tau,s_0,\omega)\geq \frac{b}{\lambda}\omega_{\ast}$ and $\phi(\tau,s_0,\omega)\leq \frac{b}{\lambda}\omega^{\ast}$. It means that for all $\tau\geq 0$ and $\omega \in \Omega$
\begin{equation*}
   \phi\left(\tau,\frac{b}{\lambda}\Omega,\omega\right) \subset  \frac{b}{\lambda}\Omega ~~\Rightarrow~~\mathcal{O}^{+}(s_0)\subset \frac{b}{\lambda}\Omega \quad\text{for all}\quad s_0 \in \frac{b}{\lambda}\Omega.
\end{equation*}
Now let's take the points $s_0,s_1 \in \text{int}(\frac{b}{\lambda}\Omega)$ and suppose w.l.o.g that $s_0 < s_1$. Since
\begin{align*}
 \phi(\tau,s_0,\omega^{\ast})=& \rme^{-\tau\lambda} \bigg(s_0-\frac{b}{\lambda}\omega^{\ast}\bigg)+\frac{b}{\lambda}\omega^{\ast} \rightarrow \frac{b}{\lambda}\omega^{\ast} \quad \text{as} \quad \tau \rightarrow +\infty   \\
 \phi(\tau,s_1,\omega_{\ast})=& \rme^{-\tau\lambda} \bigg(s_1-\frac{b}{\lambda}\omega_{\ast}\bigg)+\frac{b}{\lambda}\omega_{\ast} \rightarrow \frac{b}{\lambda}\omega_{\ast} \quad \text{as} \quad \tau \rightarrow +\infty 
\end{align*}
then there exist $\tau_0,\tau_1 >0$ such that 
\begin{equation*}
 \phi(\tau_0,s_0,\omega^{\ast})=s_1 \quad \text{and}\quad \phi(\tau_1,s_1,\omega_{\ast}) =s_0.
\end{equation*}
Hence, we have that $\mathcal{O}^{+}(s_0)=\text{int}(\frac{b}{\lambda}\Omega)$ for all $s_0 \in \text{int}(\frac{b}{\lambda}\Omega)$ and obtain the following by continuity 
\begin{equation*}
    \mathcal{O}^{+}(s_0)=\frac{b}{\lambda}\Omega \quad \text{for all } \quad s_0 \in \frac{b}{\lambda}\Omega.
\end{equation*}
Consequently, it means that $\frac{b} {\lambda}\Omega$ is a control set of $\Sigma _\mathbb{R}$. \\

$\bullet$~ Now let us show that $\Sigma_\mathbb{R}$ does not admit control sets in $\mathbb{R} \setminus \frac{b}{\lambda}\Omega=(\frac{b} {\lambda}\omega^{\ast},+\infty)\cup (-\infty, \frac{b} {\lambda}\omega_{\ast})$. For any $s_0 \in \mathbb{R} \setminus \frac{b}{\lambda}\Omega$, we obtain that
\begin{equation*}
\phi(\tau,s_0,\omega)-s_0= e^{-\tau\lambda} \Big (s_0-\frac{b}{\lambda}\omega\Big )+\frac{b}{\lambda}\omega-s_0 = \Big (e^{-\tau\lambda}-1\Big ) \Big( s_{0}-\frac{b}{\lambda}w\Big ).
\end{equation*}
This allows us to achieve the following relations
\begin{align*}
s_0 > \frac{b}{\lambda}\omega^{\ast} ~~\Rightarrow~~\phi(\tau,s_0,\omega)\leq s_0 ~~\Rightarrow~~\mathcal{O}^{+}(s_0)\backslash\{s_0\}\subset (-\infty,s_0)
\end{align*}
and
\begin{equation*}
 s_0 < \frac{b}{\lambda}\omega^{\ast} ~~\Rightarrow~~\phi(\tau,s_0,\omega)\geq s_0 ~~\Rightarrow~~\mathcal{O}^{+}(s_0)\backslash\{s_0\}\subset (s_0,+\infty).   
\end{equation*}
As a result, if taken as $s_0,s_1 \geq \frac{b}{\lambda}\omega^{\ast} $ with $s_0 < s_1$ then $\Sigma _\mathbb{R}$ has no trajectory starting at $s_1$ and approaching an arbitrary point $s_0$. It means that any control set of $\Sigma _\mathbb{R}$ contained in $\Big(\frac{b} {\lambda}\omega^{\ast},+\infty\Big)$ cannot have two distinct points. Otherwise, it contradicts the condition of being a control set. Moreover, if $\{s_0\}$ is a control set of $\Sigma _\mathbb{R}$ contained in $\Big(\frac{b} {\lambda}\omega^{\ast},+\infty\Big)$, we have the following by using the definition of control sets
\begin{equation*}
 \exists \omega \in \Omega \quad   \forall \tau \in \mathbb{R}\quad  \phi(\tau,s_0,\omega)=s_0 ~~\Leftrightarrow~~ (\rme^{-\tau\lambda}-1)\Big(s_0-\frac{b} {\lambda}\omega\Big)=0,
\end{equation*}
which is definitely impossible. Hence, $\Sigma _\mathbb{R}$ does not admit control sets in $\Big(\frac{b} {\lambda}\omega^{\ast},+\infty\Big)$. Similarly, $\Sigma _\mathbb{R}$ does not admit control sets in $\Big(-\infty,\frac{b} {\lambda}\omega_{\ast}\Big)$. Consequently, we show the uniqueness of $\frac{b} {\lambda}\Omega$. Analogously to the above steps, $\mathcal{C}_{\mathbb{R}}=\text{int}~(\frac{b}{\lambda} \Omega)$ is obtained for the case $\lambda<0$. Finally, let us prove the case $\lambda=0$. The solutions of $\Sigma _\mathbb{R}$ are constructed by concatenations of the curves $\phi(\tau,s,\omega)=b\omega\tau+s_0$ since our dynamical system is $\dot{s}=\omega b$. Pick any arbitrary points $s_1,s_2$ with $s_1 < s_2$. In this case any $\omega_1,\omega_2 \in \Omega$ with $b\omega_1>0$ and $b\omega_2<0$, then we have
\begin{align*}
\tau_1 &= \frac{s_2 - s_1}{b\omega_1}>0 ~~\Rightarrow~~ \phi(\tau_1,s_1,\omega_1)=b\omega_1 \frac{s_2 - s_1}{b\omega_1}+s_1 =s_2\\
\tau_2 &= \frac{s_1 - s_2}{b\omega_2}>0 ~~\Rightarrow~~ \phi(\tau_2,s_2,\omega_2)=b\omega_2 \frac{s_1 - s_2}{b\omega_2}+s_2 =s_1
\end{align*}
Therefore we achieve the controllability of $\Sigma _\mathbb{R}$ over the entire set of real numbers.
\end{proof}

\bigskip

The previous lemma will be used in our next result.

\begin{theorem}\label{examplo}
For the one-input LCS on $\R\times\T$
\begin{equation*}
(\Sigma_{1,1}):\quad \left\{ 
\begin{array}{l}
\dot{s}=-\lambda s+\omega b \\ 
\dot{[t]}=\frac{1}{2}\alpha s^{2}+\gamma s+\omega (c+as)
\end{array}%
\right.,   \hspace{.5cm} \omega\in\Omega ,
\end{equation*}%
where $a, b, c, \alpha, \gamma, \lambda\in\R$ and $\alpha=0$ if $\gamma\neq 0$, it holds:

\begin{enumerate}
    \item $\Sigma_{1,1}$ satisfies the LARC if and only if 
\begin{equation*}
b(2a\lambda +b\alpha )\neq 0 \quad \text{or} \quad b(b\gamma +\lambda c)\neq 0.
\end{equation*}

\item Under the LARC, the set $\mathcal{C}_{\Sigma_{1,1}}=\mathcal{C}_{\Sigma_{\mathbb{R}}} \times \T$ is the only control set of $\Sigma_{1,1}$,  where $\Sigma_{\mathbb{R}}$ is the LCS on $\R$ given by the first equation on $\Sigma_{1,1}$.
\end{enumerate}

\end{theorem}

\begin{proof} 1. Let us show that $\mathrm{span}_{\mathcal{L}A}\{\widehat{\XC}_{1,1}, \widehat{B}_{1,1}\}(s,[t])= \R^2 $ for all $(s, [t])\in \R \times \R$, where 
$$\widehat{\XC}_{1,1}(s, [t])(-\lambda s, \frac{1}{2}\alpha s^2+\gamma s)=\hspace{.5cm}\mbox{ and }\hspace{.5cm}\widehat{B}_{1,1}(s, t)=(b, c+as).$$
Firstly, looking at the Lie bracket of $\widehat{\XC}_{1, 1}$ and $\widehat{B}_{1, 1}$ we have that
\begin{align*}
 [\widehat{\XC}_{1,1}, \widehat{B}_{1,1}]=&\bigg(b\lambda ,-a\lambda s-b(\alpha s+\gamma) \bigg) \\
 =& \lambda \widehat{B}_{1,1} - \bigg\{ \underbrace{(0,2\lambda a + b \alpha)}_{:=Z_1}s + \underbrace{(0,\lambda c + b \gamma)}_{:=Z_2} \bigg\} \\
 =& \lambda \widehat{B}_{1,1} - (sZ_1+Z_2).
 \end{align*}
Then let's consider the other brackets, respectively:
\begin{align*}
 [sZ_1 , \widehat{\XC}_{1,1}] &= \lambda s Z_1  && [sZ_1 , \widehat{B}_{1,1}] = -b Z_1 \\
 [Z_2 , \widehat{\XC}_{1,1}] &=0  &&   [Z_2 , \widehat{B}_{1,1}]=0\\
 [[\widehat{\XC}_{1,1}, \widehat{B}_{1,1}],\widehat{\XC}_{1,1}] &=-\lambda^2 \widehat{B}_{1,1}+2\lambda Z_2 && [[\widehat{\XC}_{1,1}, \widehat{B}_{1,1}],\widehat{B}_{1,1}]=2bZ_1.
\end{align*}
If it is continued in this process, we see that all brackets just depend on the vector fields $\widehat{\XC}_{1,1},\widehat{B}_{1,1}, Z_1$ and $Z_2$. Finally, one can obtain that LARC satisfied if and only if $bZ \neq 0$ where $Z=sZ_1+Z_2$. Thus the proof is completed.

2. Let us assume that $\lambda >0$. For all $P \in \R\times\T$ and $\omega\in\Omega$ we write the solution of the $\Sigma_{1,1}$ as
\begin{equation*}
\phi(\tau,P,w)=\Big(\phi_1(\tau,P,w),\phi_2(\tau,P,w)\Big),
\end{equation*}
and notice that $\phi_1$ is actually the solution of the associated system $(\Sigma _\mathbb{R}):~\dot{s}=-\lambda s+\omega b$. Since we are assuming the LARC, $b\neq 0$ and by Lemma \ref{lemma}, we have that the control set $\CC_{\Sigma_{\R}}=\frac{b}{\lambda}\Omega$ is positively-invariant, implying that 
$$\OC^+(P)\subset \CC_{\Sigma_{\R}}\times\T, \hspace{.5cm}\mbox{ for all }\hspace{.5cm}P\in \CC_{\Sigma_{\R}}\times\T.$$

Let us consider the polynomial $p(\omega)$ given by 
$$p(\omega)=\frac{b}{2\lambda^2}(b\alpha+2a\lambda)\omega^2+\frac{1}{\lambda}(b\gamma+c\lambda)\omega.$$
By the LARC, $p(\omega)$ is a nontrivial polinomial with, at most, two zeros in $\Omega$.  Consider now $$P_0, P_1\in\inner(\CC_{\Sigma_{\R}}\times\T),\hspace{.5cm}\mbox{ with } \hspace{.5cm}P_1=\left(\frac{b}{\lambda}\omega_1, [t_1]\right) \hspace{.5cm}\mbox{ and }\hspace{.5cm} p(\omega_1)\neq 0.$$
Since, by Lemma \ref{lemma}, controllability holds in $\inner\CC_{\Sigma_{\R}}$, there exists a positive time $\tau_0$ and a control $\omega_0$ such that 
$$\phi_1(\tau_0, P_0, \omega_0)=\frac{b}{\lambda}\omega_1\hspace{.5cm}\implies\hspace{.5cm}\phi(\tau_0, P_0, \omega_0)=\left(\frac{b}{\lambda}\omega_1,\underbrace{\phi_2(\tau_0,P_0,\omega_0)}_{=[\hat{t}_1] \in \mathbb{T}} \right) := \hat{P_1}$$

On the other hand, a simple calculation show us that, for all $\tau\in\R$, 
$$\phi_1(\tau, \hat{P_1}, \omega_1)=\frac{b}{\lambda}\omega_1 \hspace{.5cm}\mbox{ and }\hspace{.5cm}\phi_2(\tau, \hat{P_1}, \omega_1)=\left[\hat{t}_1
+ \tau \cdot p(\omega_1)\right],$$
and hence, the assumption $p(\omega_1)\neq 0$ implies the existence of $\tau_1>0$ such that $\phi_2(\tau_2, \hat{P}_1, \omega_1)=[t_1]$ and so 
$\phi(\tau_1, \hat{P}_1, \omega_1)=P_1$.
As a consequence, if $p(\omega)=0$, we have that
$$(\CC_{\Sigma_{\R}}\times\T)\setminus(\{0, \omega\}\times\T)\subset\OC^+(P)\hspace{.5cm}\mbox{ for all }\hspace{.5cm}P\in (\CC_{\Sigma_{\R}}\times\T)\setminus(\{0, \omega\}\times\T).$$

Since $(\CC_{\Sigma_{\R}}\times\T)\setminus(\{0, \omega\}\times\T)$ is dense in $\CC_{\Sigma_{\R}}\times\T$ we conclude that 
$$\CC_{\Sigma_{\R}}\times\T=\overline{\OC^+(P)}\hspace{.5cm}\mbox{ for all }\hspace{.5cm}P\in\CC_{\Sigma_{\R}}\times\T,$$
showing that $\CC_{\Sigma_{1, 1}}=\CC_{\Sigma_{\R}}\times\T$ is a control set of $\Sigma_{1, 1}$ (See Figure \ref{fig1} (b)). Uniqueness of $\CC_{\Sigma_{1, 1}}$ follows direct from the fact that $\CC_{\Sigma_{\R}}$ is the only control set of the associated system $\Sigma_{\R}$.

Since the case for $\lambda <0$ is analogous to the previous one, let us now examine the case where $\lambda=0$. In this case, by Lemma \ref{lemma}, the control set $\CC_{\Sigma_{\R}}$ of $\Sigma_{\R}$ is the whole real line and hence, we have to show that $\Sigma_{1, 1}$ is controllable.

Similarly as the previous case, let us consider the polynomial 
$$q(s)=\frac{1}{2}\alpha s^2+\gamma s.$$
By the LARC, $q(s)$ is a nonzero polynomial. Moreover, the fact that $\gamma\neq 0\implies \alpha=0$, gives us that $s=0$ is the only root of $q$.

Take any two points $P_0=(s_0, [t_0])$ and $P_1 =(s_1,[t_1])$ in $\mathbb{R} \times \mathbb{T}$. Let us investigate the following cases for determining the trajectory from $P_0$ to $P_1$ (See Figure \ref{fig1} (a)).

\begin{itemize}
		\item[(a)] If $s_1\neq 0$, the fact that $\Sigma_{\R}$ is controllable, assures the existence of $\omega_0$ and $\tau_0$ such that 
		$$\phi(\tau_0, P_0, \omega_0)=(s_1, [\hat{t}_1])=:\hat{P}_1.$$
		
		By considering the control $\omega=0$, we have that 
		$$\phi(\tau, \hat{P}_1, 0)=(s_1, \left[\hat{t}_1+\tau \cdot q(s_1)\right]).$$
		As a consequence, there exists $\tau_1>0$ such that $\phi(\tau, \hat{P}_1, 0)=(s_1, [t_1])$, showing, by concatenation, that we can reach $P_1$ from $P_0$, when $s_1\neq 0$.

\item[(b)] If $s_1=0$, take $\omega\neq 0$ and $\tau'>0$. Then,
$$P'_1:=\phi_1(-\tau, P_1, \omega)\hspace{.5cm}\mbox{ satisfies } \hspace{.5cm}\phi(\tau, P'_1, \omega)=P_1,$$
and the first component of $P'_1$ is $s_1'=-\tau b\omega\neq 0$. By the previous item, there exists a trajectory connecting $P_0$ and $P_1'$ and hence, we can connect $P_0$ to $P_1$, concluding the prove.
\end{itemize}

\begin{figure}[H] 
	\centering
	\includegraphics[scale=.6]{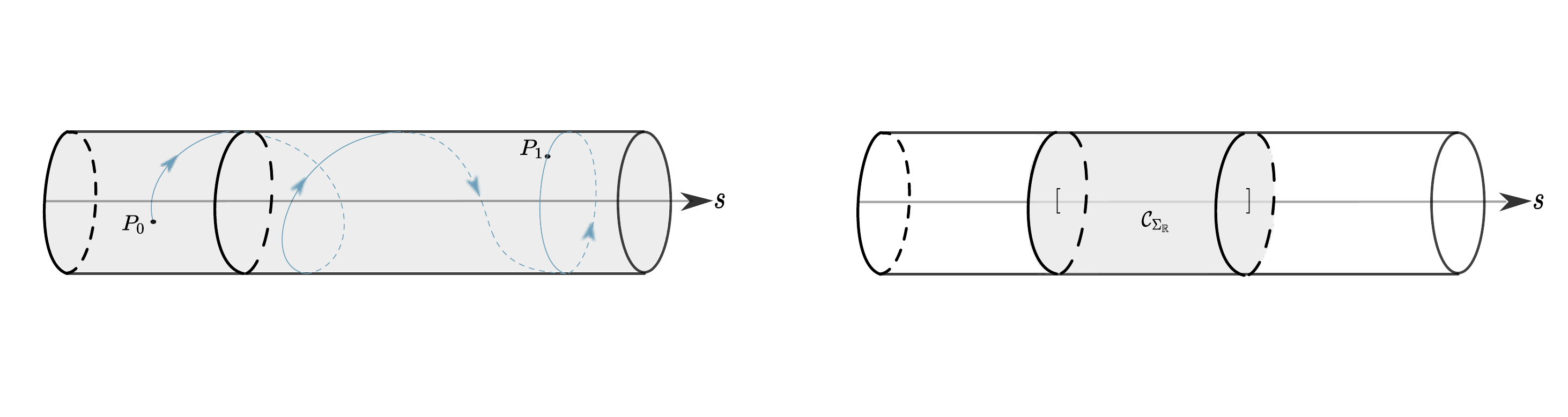}
 \vspace*{-10mm}
	\caption{(a) Trajectory connecting $P_0$ and $P_1$\hspace{4cm} (b) Control set for nonzero $\lambda$.\hspace{-3cm}}
	\label{fig1}
\end{figure}

\end{proof}

\end{document}